\documentclass[a4paper]{article}
\usepackage[margin=25mm]{geometry}
\usepackage{amsmath}
\usepackage{amsfonts}
\usepackage{amssymb}
\usepackage{amsthm}
\usepackage{graphicx}
\usepackage{dsfont}
\usepackage{float}
\usepackage[normalem]{ulem}
\usepackage{verbatim}
\usepackage{color}
\immediate\write18{texcount -tex -sum  \jobname.tex > \jobname.wordcount.tex}
\usepackage[none]{hyphenat}

\newtheorem{prop}{Proposition}

\title{Phenotype divergence and cooperation in isogenic multicellularity and in cancer}
\author{Frank Ernesto Alvarez$^1$ \& Jean Clairambault$^2$\\
\small $^{1}$INSA Toulouse, France // orcid 0000-0002-6651-7374 \\
        \small $^{2}$Inria and LJLL, Sorbonne Universit\'e, Paris, France // orcid 0000-0002-8336-9641}
\date{\today}

\begin{document}
\sloppy
\maketitle

\begin{abstract}
We discuss the mathematical modelling of two of the main mechanisms which pushed forward the emergence of multicellularity: phenotype divergence in cell differentiation, and between-cell cooperation. In line with the atavistic theory of cancer, this disease being specific of multicellular animals, we set special emphasis on how both mechanisms appear to be reversed, however not totally impaired, rather hijacked, in tumour cell populations. Two settings are considered: the completely innovating, tinkering, situation of the emergence of multicellularity in the evolution of species, which we assume to be constrained by external pressure on the cell populations, and the completely planned - in the {\it body plan} - situation of the physiological construction of a developing multicellular animal from the zygote, or of  bet hedging in tumours, assumed to be of clonal formation, although the body plan is largely - but not completely - lost in its constituting cells. We show how cancer impacts these two settings and we sketch mathematical models for them. We present here our contribution to the question at stake with a background from biology, from mathematics, and from philosophy of science.

\end{abstract} \hspace{10pt}

{\it Keywords}. {Differentiation, cooperation, multicellularity, cancer disease, structured population models, philosophy of science}

\section{Biological and evolutionary-developmental background}

\subsection{Being or not teleological: the two settings considered}

Although this may seem completely trivial to state, let us emphasise that for us there is no such thing as teleology, i.e., orientation in a given direction or towards a given goal, in the general evolution of multicellular animals, which is constituted of a succession of haphazard strategic choices of adaptation to changing environments in existing evolutionary units, at one stage of evolution towards an identified next one. Such adaptations, often resulting in branchings of clades, as solutions to existential problems, imposed by external constraints as stresses~\cite{NedelcuMichodBioessays2020, WagneretalBioessays2019} induced by changes in the environment, are by no means unique, admitting that evolution proceeds by trials and errors, and by tinkering~\cite{JacobScience1977} from available material to solve such problems. We proposed in~\cite{ACC2022} a mathematical scheme to model the phenotypic divergence that may be a basis for such environmental stress-induced evolutionary steps. 

Conversely, teleology is of course present in the embryonic development of multicellular animals, which, according to Haeckel's formula ``Ontogeny recapitulates phylogeny''~\cite{Gould1977, Haeckel1866}, follows in each species the evolutionary choices made at each branching step of the evolution of species, leading from the fecundated egg (most frequent form of elementary material evolutionary unit in multicellular animals~\cite{WolpertSzathmaryNature2002}, those who are subject to cancer~\cite{Aktipis2015, NedelcuBST2020}) to adult animals with their completely differentiated cell types, following the {\it body plan}~\cite{Davidson1995, MulleretalIntRevCytol2004} characteristic of the species. From this holistic point of view, evolution of species is nothing but evolution of the body plan, evolution of genes and of gene regulatory networks being completely dependent upon this master regulator. We suggest here that understanding the cooperation principles that have been optimised (noting that an optimisation problem may have diverse solutions) at each developmental step may benefit from a close look at the mechanisms of the evolutionary steps that have determined the species body plan, and we sketch mathematical ways to achieve this task. 

One of the main difficulties in understanding and representing the design of the body plan is how to introduce mechanisms of coherence (for signals) and cohesion (for tissues) that make a multicellular organism stable and functional, with compatibility and cooperation between tissues and organs, and we are aware of the fact that such complete understanding still lies ahead of us. However, localised absence of coherence between tissues of an organism by lack of control on differentiations is precisely the main characteristic of cancer, the second and in our opinion resulting from the first one, being absence of control on proliferation~\cite{Bertolasobook2016}. We propose that evolution of cooperation between cells, that has been identified in tumours~\cite{ClearyNature2014, PolyakMarusykNature2014, TabassumPolyakNature2015}, is a reactivation of mechanisms present in the body plan that are still present, although chaotic, uncontrolled and doomed to fail at the level of the organism, in tumour cells, may rely on elementary evolutionary mechanisms that have been designed in the evolutionary past of their body plan, so that this point should be better understood to efficiently represent cooperation in tumours.

\subsection{The atavistic theory of cancer}
Recently popularised by physicists Paul Davies and Charles Lineweaver, together with oncologist Mark Vincent, the atavistic theory of cancer~\cite{DaviesLineweaver2011, LineweaveretalBioessays2014, LineweaverDavies2020, Lineweaveretal2021, VincentBioessays2011}, had in fact been envisioned already in 1996 by oncologist Lucien Israel~\cite{IsraelJTB1996}, and likely as early as 1914 by biologist Theodor Boveri~\cite{Boveri1914}, although none of these scientists seem to have been initially aware of the works of their predecessors. It helps us understand tumour progression and intratumoral organisation from a long-term evolutionary viewpoint. Briefly, it relies on the ideas that 1) all cancer cells are multicellular animal cells, results of a billion year-old evolution from unicellular organisms, and as such keep in their genomes powerful remnants of the organismic defence and construction mechanisms borne in their body plans (even if this term is not used by Davies and Lineweaver, they only mention their genomes); 2) tumours are results of a regression in the development of the organism, corresponding to early, incoherent versions of {\it ``an ancient genetic toolkit of pre-programmed behaviors''}, which we may freely identify as an unachieved evolutionary version of the species body plan, and which they name ``Metazoa 1.0''. The atavistic theory thus clearly states that a tumour is not just the result of some aberrant stochastic mutation in somatic cells (the somatic mutation theory, SMT, recently reviewed and compared to the atavistic theory in~\cite{LineweaverDavies2020}), but that it rather follows predictable paths in such regression towards a poorly organised, incoherent population of cells, nevertheless constituted of animal cells that are highly plastic (and thus resistant to external therapeutic pressure by anticancer drugs), as they have the power to differentiate and de-differentiate, and also to loosely cooperate between them in tumours. The works of David Goode and colleagues~\cite{Trigos2017, Trigos2018, Trigos2019, Trigos2023} have evidenced in cancer samples silencing of genes of multicellularity and compatibility between expression of genes of multicellularity and of unicellularity, resulting in escaping organismic control on cell differentiation (in other words, developing cell plasticity) and on proliferation, tending to a widely autonomic behaviour which is a characteristic of cells in tumour tissues.

The atavistic theory of cancer is little by little, as more evidence in the study of ancient genes becomes known and published~\cite{Trigos2017, Trigos2018, Trigos2019, Trigos2023}, gaining recognition among theorists of cancer biology, however still quite limited in the field of oncology, where people question its amenability to produce innovations in the therapeutics of cancer. Innovating theories may take a long time to reverse the argument of ``authority  of tradition''~\cite{BayleComete1682}. The present situation may remind us, {\it mutatis mutandis}, of the way geographers received in 1912 with much skepticism Alfred Wegener's theory of continental drift~\cite{Wegener1912}, until it was completely justified fifty years later by the theory of plate tectonics and progressively admitted by all geophysicists. A limitation to a wider acceptance of the atavistic theory is the present lack of sufficient evidence susceptible to convince biologists and philosophers of cancer, who prefer to keep on the ``safe'' side of science under development and, at least temporarily, reject it as not sufficiently relying on facts. Indeed, when it is mentioned in recent texts of philosophy of science - by authors who nevertheless must be commended for at least mentioning it -, the atavistic theory of cancer is not always correctly summed up, sometimes even presented in an off-hand way with arguments against it that show but partial understanding, as in~\cite{Pradeuetal2023}. A mere hypothesis, really? At least a uniting one in understanding cancer, fully compatible with the holistic point of view on evolution that we have mentioned above.

\subsection{Why and how does multicellularity fail in cancer?}

Cancer is thus, taking the atavistic theory of cancer for granted - although it tells us nothing about the very origin of the disease -, the progressive result of a failed maintenance of the teleological (or teleonomical, if one wants to explicitly exclude any intentionality, which is our position) construction of an animal. It may be described as essentially ``a deunification of the individual''~\cite{Pradeu2019}. In the perspective of evolved multicellularity, it is tempting to describe - an epistemological position we assume - such material construction at the level of genes and gene regulatory networks, initially not from the zygote, but from nonclonal colonies of cells (i.e., before the invention of the egg~\cite{WolpertSzathmaryNature2002} and of the {\it body plan} contained in it) in three successive steps.

At the first step, the colony level, exist only genes of the cell division cycle and cell death, likely by quorum sensing. At the second step are introduced genes coding for transcription factors and (unregulated) differentiation. At the third step appear genes coding for epigenetic regulations, the top level of fine local regulations, that are themselves subject to central regulations in higher-level animals such as bilaterians. Such hierarchy is remarkably found, in a reverse order, in the evolution in malignancy found in fresh blood samples of patients with acute myelogenous leukaemia~\cite{HirschetalNatureLett2016}, which induces us to propose a scenario for cancer progression as relying firstly on epigenetic gene alterations (which includes differentiation control), secondly on alterations in differentiation, and only very late on alterations in cell cycle regulations, which are the strongest basis of proliferation. Unfortunately so far, with the remarkable and recent exception of the successes of immunotherapy, cancer therapies target mainly this strength~\cite{LineweaveretalBioessays2014}.

Let us in the sequel consider the question of the dynamic behaviour, in an already constituted multicellular organism, of cancer cell lines as compared to healthy cell lines. As regards adaptation to changes in environmental pressure, for healthy and for cancer animal cell lines, representative cells of both types of lines that conserve in their genome for a very long past, at the time scale of Darwinian evolution, the same atavistic programme of the species body plan, the adaptative scenario is the same. It is indeed deterministic, however with easy bet hedging (resorting to atavistic adaptive varied scenarios that are normally repressed, fixed by cohesion rules, in cohesive multicellular animals) and very fast adaptation due to their high plasticity, in the case of cancer cells.

In the case of healthy cells, differentiation is strictly controlled for the sake of organismic cohesion, so that short-time, i.e., cell life-time, adaptation by non-genetic (epigenetic) ways is weak and slow (which is the same at short term). No time is left for genetic fixation of adaptive traits in healthy cells (which is not the case at the billion-year time scale of Darwinian evolution), so that healthy cell lines may be considered as evolutionary stable in an organism life-time perspective.

In the case of the very plastic cancer cells - due to poor control on their differentiation, in our and in Marta Bertolaso's view \cite{Bertolasobook2016}, one main cause of cancer -, adaptation is on the contrary fast in a life-time perspective, and so is mutational genetic fixation due to poor control on cell cycle gating. Indeed, in the cancer case, added to the deterministic and atavistic basis of the body plan with added bet hedging, may come stochasticity (e.g., due to error-prone DNA  polymerases, {\it mutatis mutandis} as shown in starving bacteria \cite{kivi}), and poor control on the quality of DNA in the cell division cycle checkpoints, inducing the high mutation rate observed in cancer lines as compared with healthy cell lines.

What are the respective parts of determinism and stochasticity in the evolutionary capacities of cancer cell lines remains to be determined. In this respect, it is noteworthy that according to Marta Bertolaso, poor control on differentiation and on proliferation in a parallel way - or is it consequential? Indeed, cellular stress resulting from alterations of control on differentiation might be a cause of poor control on proliferation as mentioned above about bacteria \cite{kivi} - are the two main traits of cancer cells.

In other words, both the epigenetic deterministic scenario of the body plan relying on differentiations in isogenic cells, however poorly controlled and inefficient in producing cohesion, and the genetic stochastic scenario of Darwinian evolution by gene mutation, however with a tremendously enhanced speed of `economic' genetic fixation at a cell lifetime scale, after first and costly - in terms of the energetic cell machinery - epigenetic adaptation, are concerned in the dynamic behaviour of cancer cell populations.

\subsection{A narrative of long-term evolution and cancer, freely exposed to the fire of philosophy of science}

We need not justify any given evolutionary path that led to such and such animal, and rather see paths followed in evolution as diverse evolutionary strategies adapted to external constraints that imposed changes on the behaviour of the actors of the evolutionary paths at stake. Let us mention here that we hold, from our point of view, which resorts to functional, physiological and anatomical evolution, these actors, or evolutionary units, to be the {\it body plans}~\cite{Davidson1995, MulleretalIntRevCytol2004} of multicellular animals, and not the individual genes, nor the gene regulatory networks that are mere effectors of evolutionary strategies, not determinants, and are only secondarily affected by them, as reflected in observations. A paleoanthropological analogy in evolution, {\it mutatis mutandis}, of such strategies at the level of divergence from a common ancestor in the Hominin lineage between Paranthropus and early Homo, relying on different dietary choices, may be found in~\cite{BalteretalNatureLett2012}. Such haphazard strategical choices in long-term, Darwinian, evolution, that have become fixed in the body plan of animal species by genetic mutations and success in species fitness, may fail in cancer, as described in the previous section.

These firstly non determined (tinkered~\cite{JacobScience1977}) strategies led to epigenetic modifications (aka epimutations), later to fixed mutations of the genes coding for the epigenetic enzymes that determine these epigenetically defined strategies yielding functional body plans, that are the bases of physiology and anatomy construction in multicellular animals. Cancer cannot change the body plan of an animal in that of another animal, and it is certainly not a new form of life. However, by loss of organismic control on differentiations, it can reverse a cohesive body plan in a given species to some intermediate, poorly defined, unachieved form of the body plan of this species, yielding a collection of still very plastic cells, in other words a tumour, or a Metazoan 1.0 in the words of the atavistic theory of cancer~\cite{DaviesLineweaver2011}. The causes of such loss of control on differentiations are unknown, and the atavistic theory tells us nothing about them. However they may consist of an abrupt change in the environmental pressure on the tissue at stake, but also may be identified as due to a mutation in the genes responsible for epigenetic control~\cite{HirschetalNatureLett2016}.

\section{Cell differentiation and phenotype divergence}

\subsection{Heterogeneity and plasticity with respect to what?}
Cell populations, healthy and cancer, are heterogeneous w.r.t. various {\it continuous} traits under study, that are used to describe their biological variability, such as cell size, age in the cell division cycle, expression of genes of drug resistance, or more functional and abstract traits determining cell population fate such as viability, fecundity, motility, plasticity, according to the biological question at stake. Plasticity~\cite{JCISMCO2020, JCECC2023} in a given trait is its capacity to change under the pressure of external constraints, such as drugs, and it has long been recognised as as relying on epigenetic factors~\cite{McCullough1998}. Plasticity may be considered as a speed of evolution from one trait distribution to another one when the surrounding environment of the cell population changes, slowly or abruptly. Such evolution may be accelerated in equations by terms of advection (especially when abrupt changes in the environment force the cell population to adapt quickly) and diffusion  (representing uncertainty in phenotype determination).

Differentiation in cell lineages, such as the ones constituting the paths of haematopoiesis, may consist either of simple maturation, following the same line towards a terminally differentiated cell type, such as the different granulocytes (neutrophils, eosinophils, and basophils) among white blood cells, or of branching, e.g., in haematopoiesis from pluripotent haematopoietic stem cells to myeloid versus lymphoid progenitors. Phenotype divergence is the biological phenomenon by which branching occurs between precursors of terminally differentiated cell types. The first identified phenomenon relying on phenotype divergence in evolution from unicellularity towards multicellularity was likely the separation between germinal cells (the germen) and germen-supporting somatic cells (the soma), proposed in 1892 by August Weismann~\cite{Weismann1892} and later mentioned by John Maynard Smith and E\"ors Szathm\'ary as the first step from unicellularity towards multicellularity, one of the major transitions in evolution~\cite{MaynardSmithSzathmary1995}. Basis of heterogeneity in cell populations within a cohesive multicellular individual, or within a tumour, phenotype divergence necessarily relies on phenotype plasticity, and it is the phenomenon we here tackle to represent in phenotype-structured equations. 

\subsection{Long-term evolution as genetic adaptation of the body plan in animals}

As mentioned in the introduction, we consider that the fundamental evolutionary unit in the great Darwinian evolution of animals is the body plan~\cite{Davidson1995, MulleretalIntRevCytol2004}, which is virtually (as it is abstract, indeed as a plan, self-developing, written as a self-extracting archive in genetic code, its dynamic extraction occurring continuously during the process of animal development) present in every physiologically complete nucleated animal cell, starting from the zygote, i.e., the initial fecundated egg. The genes and gene regulatory networks that materially proceed from it and serve to design and cohesively maintain the construction of the animal when it is achieved, are its observable materialisation.

Anatomically in 3D observations, physiologically by the observation of the great functions of the organism, and genetically by investigation the genes that have been identified (e.g., by KO experiments) in different species to correspond to anatomic structures and physiological functions, and their expression, we may have access to material reflections of the body plan, and thus partially reconstitute its evolution across species. This is precisely what has been investigated about the genes at the origin of multicellularity and their correspondence with the genes that are altered in cancer by Domazet-Lo\v so and Tautz~\cite{domazet2008, domazet2010}, and later by Trigos {\it et al.}~\cite{Trigos2017, Trigos2018, Trigos2019, Trigos2023} in David Goode's team, giving rise and genetic arguments to the atavistic theory of cancer~\cite{DaviesLineweaver2011, LineweaveretalBioessays2014, LineweaverDavies2020}.

\subsection{A nonlocal phenotype-structured cell population model}

The reaction-diffusion-advection model proposed in~\cite{ACC2022} to exemplify {\it bet hedging} as a `tumour strategy' to diversify its phenotypes in response to deadly stress (e.g., by cytotoxic drugs), but also to represent phenotypic divergence in evolution towards multicellularity, runs as follows.\\
\medskip

Let $D=\Omega\times[0,1]$, where $\Omega:=\{C(x,y)\leqslant K\}$ (a constraint between competing traits $x$ and $y$) and $\theta\in[0, 1]$. The evolution with time $t$ of a plastic cell population of density $n(z,t)$ structured in a 3D phenotype $z=(x,y,\theta)$, where $x$=viability, $y$=fecundity, $\theta$=plasticity, with $r(z)$ and $d(z)$ growth and death rates, is given by

\begin{equation}
    \partial_t n+\nabla\cdot\Big(Vn-A(\theta)\nabla n\Big) = (r(z)-d(z)\rho(t))n,\\ \label{PheD}
 \end{equation}
 \begin{equation}
   \text{with} \;(Vn-A(\theta)\nabla n\Big)\cdot\mathbf{n}=0\mbox{ for all }z\in\partial D\;\;
   \text{($\mathbf{n}$ is a normal vector to $\partial D$)},\;\;n(0,z)=n_0(z)\mbox{ for all }z\in D,\mbox{ } \nonumber
\end{equation} where $\Omega=\{(x,y)\in[0,1]^2:(x-1)^2+(y-1)^2>1\}$, and the diffusion matrix is\\

$A(\theta)=\begin{pmatrix}
a_{11}(\theta)&0&0\\
0&a_{22}(\theta)&0\\
0&0&a_{33}
\end{pmatrix},$
with $a_{11}$ and $a_{22}$ non-decreasing functions of $\theta$,
influencing the speed at which non-genetic epimutations occur, otherwise said, it is a representation of how the internal plasticity trait $\theta$ affects the non-genetic instability of traits $x$ and $y$, by tuning the diffusion term $\nabla\cdot\{A(\theta)\nabla n\}$; the advection term
$$
\nabla\cdot \{V(t,z)n\}=\nabla\cdot\{(V_1(t,z),V_2(t,z),V_3(t,z))n\}
$$
represents the cellular stress exerted on the population by external evolutionary pressure, i.e., by changes in the cell population environment, here chosen as tearing apart the cell population between competing traits $x$ (viability) and $y$ (fecundity); and $\rho(t)=\displaystyle\int\limits_{D}n(t,z)dz$
stands for the total mass \vspace{-.1cm}of individuals in the cell population at time \vspace{.2cm}$t$.\\
\bigskip

The existence and uniqueness of solutions is obtained in finite time in a constructive way by using the compactness of a sequence of numerical solutions, which are the result of the algorithms used to discretise the model. Simulations may be obtained with instances of the functions used in the equations. For instance, to obtain phenotypic divergence (which we take as the basis of both bet hedging in cancer and of emergence of multicellularity in evolution), we consider 
over the domain $D=\Omega\times[0,1]$
an initial density given by
\[
n_0(z)=a\mathds{1}_{\{f(z)<1\}}e^{-\frac{1}{1-f(z)}},
\]
with $f(z)=\frac{\|z-z_0\|^2}{(0.025)^2}$, where $z_0=(0.25,0.25,0.5)$ and $\|\cdot\|$ is the euclidean norm. We choose the value of $a$ in such a way that $\rho_0=\int_{D}n_0(z)=1$.\\
We set the growth rate and the death rate as
\begin{align*}
    r(x,y,\theta) =\mathds{1}_{\{y>x\}}e^{-(0.1-x)^2-(0.9-y)^2}+\mathds{1}_{\{x\geqslant y\}}e^{-(0.1-y)^2-(0.9-x)^2},\;\;\;
    d(x,y,\theta) =\frac{1}{2}.
\end{align*}
This choice of growth and death rate is meant to represent the fact that different configurations of traits can be equally fitted, even to the point where the survival rate can be maximised in multiple ways.\\
We choose the diffusion matrix
\[
A(\theta)=\begin{pmatrix}
(\theta+1)10^{-6}&0&0\\
0&(\theta+1)10^{-6}&0\\
0&0&10^{-6}
\end{pmatrix},
\] 
so a higher plasticity will directly imply a higher mutation rate for the other two traits; and the advection term, tearing apart traits $x$ and $y$, is chosen as $V(t,z)\!=\!10^{-3}(-y,-x,-(x+y))$, or $10^{-3}\theta(-y,-x,-(x+y))$ if we want plasticity $\theta$ to impinge also on the advection term, representing in all cases the influence of the tumour ecosystem on the tumour cell population.\\
\begin{figure}[H]
    \centering
    \hspace{-2cm}\includegraphics[width=15cm]{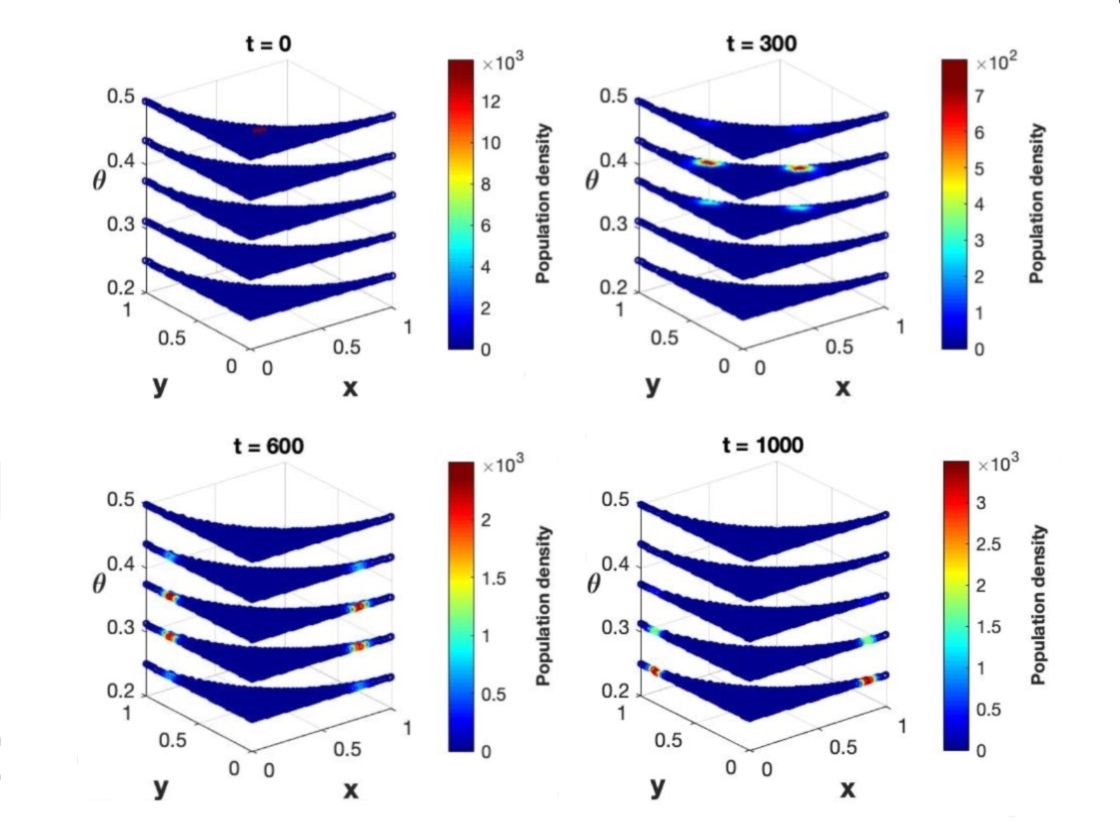}
    \caption{Phenotype divergence and loss of plasticity. On these cartoon-like figures, one can follow the progressive distancing of an initial cell population arbitrarily set at $z=(0.25, 0.25, 0.5)$, submitted to an advection gradient that tends to split the cell population into two subpopulations migrating towards the two extreme points $(0,1)$ and $(1,0)$ of the domain $\Omega$, while the plasticity variable $\theta$ decreases towards $0$.}
    \label{fig1}
\end{figure}
The reader is sent to~\cite{ACC2022} for more detailed explanations and illustrations showcasing other biological strategies and theoretical results being replicated by means of numerical simulations.

\subsection{What this model tackles and what it leaves unexplained}

Our reaction-diffusion-advection equations give the most important part in modelling phenotype divergence to the drift (advection) term representing environmental pressure from the ecosystem towards separation of phenotypes. Plasticity is naturally already present in the reaction term of this continuous phenotype-structured cell population model of adaptive dynamics, and the diffusion term adds to phenotype adaptability by uncertainty in its determination. Nevertheless, the sensitivity of phenotype adaptation and the trade-off we set between the supposed contradictory 1D phenotypes is mainly represented by the advection term and the bounded region within which the phenotypes evolve, that together represent constraints and offer possibilities of trade-offs between the phenotypes. 

This model is clearly a mathematical abstraction that may be applied as such to every possible branching situation in the physiological development of multicellular animals or in bet hedging of phenotypes in tumours. For instance, one could model more precisely in glioblastoma cells such branching situations as the ``go-or-grow'' alternative between enhancing a proliferation potential (fecundity) and a motion potential (motility)~\cite{Hatzikirou2010}, which would need to represent in the same kind of model the biological mechanisms that account for them, and about the constraints (likely of energetic nature) between them. This would help us design more precisely the advection term and the domain in phenotype space within which phenotypes evolve.  It would imply  efficient transdisciplinary collaboration on this subject between mathematicians and biologists of cancer, which we hope to develop in the future.

\section{Cooperation}

\subsection{Tinkered cooperation in the emergence of multicellularity vs. directed cooperation in constituted multicellular animals}

Noting that the question of cooperation and of division of labour has been considered by many authors at different stages of associations between individuals, including animal societies~\cite{MaynardSmithSzathmary1995}. To follow again the metaphor of the separation in evolution between Paranthropus and early Homo, the situation with respect to phenotype divergence between body plans of animals is as if, {\it mutatis mutandis}, in evolution from their common hominin ancestor,  Paranthropus and early Homo, after their genetic separation starting by fixation of initial epigenetic haphazard strategic adaptive choices (since evolution under changes in environmental pressure proceeds by tinkering~\cite{JacobScience1977}), had found interest in developing mutualistic interactions, living in symbiosis, less and less independently of one another. However, since the Paranthropus species eventually became extinct, likely due to climate changes incompatible with his too specialised vegetalian diet, whereas Homo survived, having adapted his diet to meat eating, this was actually not the case, or not in a permanent way, in the evolution of hominins.

We are aware of the fact that this metaphor is by no means perfect, and that reversible development, of epigenetic nature, within an isogenic individual (or a tumour) is not the same process as evolution of species, which is based on fixed, irreversible, genetic separations by branchings. Nevertheless, hypothesising that genetic specialisation is likely to begin with reversible epigenetic phenotype divergence before being fixed by gene mutations, we hope that it sheds some light on the processes that are at work in elementary steps in the evolution towards multicellularity and in bet hedging in tumours.

Cooperation between populations of cells resulting from such phenotype divergence may be considered as the glue that holds together all cell subpopulations in an isogenic multicellular organism. It may occur when mutualistic interactions are beneficial for all the interacting cell populations, provided that none of them becomes extinct. And it may also not occur, in which case no trace of such missed mutualism is found in the evolution of body plans. It is indeed, in our representation, the body plan that has kept memory, in each species, in constitutive intercellular gene regulatory networks, of the proper strategic choices w.r.t. phenotype divergences that lead to the design of an anatomically and physiologically cohesive animal. No tinkering is present anymore in these programmed choices designed in the body plan, and this is what we would like to represent now.

We will present two different possible approaches to the study of evolution of cooperation. The first one takes the prisoner's dilemma as a starting point, and considers reciprocity as a factor influencing the strategies of both players. The possible outcomes for a long running game are studied, and finally, a way to model a scenario with $n$ players is described. The second modelling choice is through an integro-differential system structured according to the probability of cooperation. In this case, reciprocity is represented by an advective term. For a simple set of hypotheses we show that cooperation might mark the difference between extinction or proliferation for two interacting populations.

\subsection{Prisoner's dilemma and reciprocity}\label{Toy}
According to \cite{Axelrod}, an initial intention for cooperation and the existence of reciprocity are crucial for the evolution of cooperation, even in an environment composed of egoistic individuals. However, one may wonder what are the conditions that guarantee this to be true; after all, it can be expected that, if reciprocity is stronger in the absence of cooperation, then cooperation becomes less usual. In other words: when is reciprocity a catalyst for cooperation? The following (very simple) model tackles this question.\\
Consider two players (that can range from cells to entire groups of individuals, such as governments) involved in the repeated prisoner's dilemma game. Player $A$ will initially cooperate with probability $p_0>0$ while player $B$ will do so with probability $q_0>0$. We assume both values to be strictly positive to account for the initial intention of cooperation described in \cite{Axelrod}. Both players will modify their probabilities of cooperation at turn $k+1$ (denoted as $p_{k+1}$ and $q_{k+1}$ respectively) by following the rule:
\[
p_{k+1}=
\left\{\begin{matrix}p_{k}+\varepsilon_{11}(1-p_k),& \mbox{ if player B cooperated in turn }$k$,\\
&\\
p_{k}(1-\varepsilon_{12}),& \mbox{ if not},
\end{matrix}\right.
\]
and
\[
q_{k+1}=
\left\{\begin{matrix}q_{k}+\varepsilon_{21}(1-q_k),& \mbox{ if player A cooperated in turn }$k$,\\
&\\
q_{k}(1-\varepsilon_{22}),& \mbox{ if not},
\end{matrix}\right.
\]
where $0<\varepsilon_{ij}<1$ for $i,j\in\{1,2\}$. According to this model, both players modify their strategy by ``learning'' from each other. A different strategy was already studied in \cite{Murase}, where players could modify their strategy by imitation.\\
We recall that the payoff matrix of the prisoner's dilemma game is given by
\[
\begin{pmatrix}
b-c&-c\\
b&0
\end{pmatrix},
\]
where $b$ is the benefit and $c$ is the cost of cooperation ($b>c$). Hence, the expected gain for players $A$ and $B$ at turn $k$ are given by
\[
E_A^k=(b-c)p_kq_k+b(1-p_k)q_k-cp_k(1-q_k)=bq_k-cp_k\mbox{ and }E_B^k=bp_k-cq_k,
\]
respectively. Therefore, the average expected gain at turn $k$ is given by the relation
\[
E_k=\frac{(b-c)}{2}(p_k+q_k).
\]
Given that the probability of both players cooperating at turn $k$ is equal to $p_kq_k$, our interest falls then on the question: What are the conditions over the values $\varepsilon_{ij}$, $i,j\in\{1,2\}$, such that the sequence $(p_k,q_k)$ converges towards a non trivial limit ? In such cases, when does the average expected gain can be expected to increase ? \\

In order to answer these questions we first explicitly give the values of $p_{k+1}$ and $q_{k+1}$ as functions of $p_k$ and $q_k$. Thanks to the law of total probability, we get the relations
\begin{align*}
    p_{k+1}&=q_k(p_{k}+\varepsilon_{11}(1-p_k))+(1-q_k)p_{k}(1-\varepsilon_{12})\\
    &=(1-\varepsilon_{12})p_k+\varepsilon_{11}q_k+(\varepsilon_{12}-\varepsilon_{11})p_kq_k=:f_1(p_k,q_k),\\
    q_{k+1}&=p_k(q_k+\varepsilon_{21}(1-q_k))+(1-p_k)q_k(1-\varepsilon_{22})
    \\
    &=(1-\varepsilon_{22})q_k+\varepsilon_{21}p_k+(\varepsilon_{22}-\varepsilon_{21})p_kq_k=:f_2(p_k,q_k).
\end{align*}
If this sequence has a limit $(p^*,q^*)$, it must satisfy the relation 
\begin{equation}
\left\{
\begin{matrix}
    p^*&=&f_1(p^*,q^*),\\
    &&\\
    q^*&=&f_2(p^*,q^*).
\end{matrix}
\right.\label{Invariant}
\end{equation}
In the following proposition we will identify the possible values for $(p^*,q^*)$ and determine their stability.

\begin{prop}\label{Stability}
Consider a couple $(p_0,q_0)$ and the value $e=\varepsilon_{11}\varepsilon_{21}-\varepsilon_{12}\varepsilon_{22}$. 
\begin{itemize}
    \item[i)] If $e<0$, then the only possible values for $(p^*,q^*)$ are $(0,0)$ and $(1,1)$. The first one is a stable fixed point and the second one is an unstable fixed point.
    \item[ii)] If $e>0$, then the only possible values for $(p^*,q^*)$ are $(0,0)$ and $(1,1)$. The first one is an unstable fixed point and the second one is an stable fixed point.
    \item[iii)] If $e=0$ then $(p^*,q^*)$ is the unique solution of
    \begin{align*}
    \varepsilon_{22}p_0+\varepsilon_{11}q_0&=\varepsilon_{22}p^*+\varepsilon_{11}q^*,\\
    q^*&=\frac{\varepsilon_{12}p^*}{\varepsilon_{11}+(\varepsilon_{12}-\varepsilon_{11})p^*},
\end{align*}
and it is a stable fixed point.
\end{itemize}
\end{prop}
\begin{proof}
 Notice that the values $(0,0)$ and $(1,1)$ are always a solution of \eqref{Invariant}. The stability of said fixed points (and others we wil determine) can be study by means of the eigenvalues of the Jacobian matrix.
\paragraph{First case: The unbalanced scenario: ($\varepsilon_{11}\varepsilon_{21}\neq\varepsilon_{12}\varepsilon_{22}$)}
If this condition is satisfied, a simple computation shows that there are not non-trivial solutions for \eqref{Invariant}. Hence, if a limit exists, it has to be either the $(0,0)$ or the $(1,1)$. The Jacobian matrix of the system at each of these points is equal to
\[
J_0:=J(0,0)=\begin{pmatrix}
    1-\varepsilon_{12}&\varepsilon_{11}\\
    \varepsilon_{21}&1-\varepsilon_{22}
\end{pmatrix}\mbox{ and }J_1:=J(1,1)=\begin{pmatrix}
    1-\varepsilon_{11}&\varepsilon_{12}\\
    \varepsilon_{22}&1-\varepsilon_{21}
\end{pmatrix}.
\]
The eigenvalues of $J_0$ are then
\[
\lambda^{0}_{1}=\frac{2-(\varepsilon_{12}+\varepsilon_{22})-\sqrt{(\varepsilon_{12}+\varepsilon_{22})^2+4e}}{2} \mbox{ and }\lambda^{0}_{2}=\frac{2-(\varepsilon_{12}+\varepsilon_{22})+\sqrt{(\varepsilon_{12}+\varepsilon_{22})^2+4e}}{2},
\]
while those of $J_1$ are
\[
\lambda^{1}_{1}=\frac{2-(\varepsilon_{11}+\varepsilon_{21})-\sqrt{(\varepsilon_{11}+\varepsilon_{21})^2-4e}}{2} \mbox{ and }\lambda^{1}_{2}=\frac{2-(\varepsilon_{11}+\varepsilon_{21})+\sqrt{(\varepsilon_{11}+\varepsilon_{21})^2-4e}}{2}.
\]
If $e<0$, then $-1<1-(\varepsilon_{12}+\varepsilon_{22})<\lambda^0_1<\lambda^0_2<1$ and $\lambda^1_2>1$, hence $(0,0)$ is stable and $(1,1)$ is unstable. On the other hand, if $e>0$, then $\lambda^0_2>1$ and $-1<1-(\varepsilon_{11}+\varepsilon_{21})<\lambda^1_1<\lambda^1_2<1$, hence $(0,0)$ is unstable and $(1,1)$ is stable.
\paragraph{Second case: The balanced scenario ($\varepsilon_{11}\varepsilon_{21}=\varepsilon_{12}\varepsilon_{22}$)} 
Under this condition, it is straightforward to notice the relation
\[
\varepsilon_{22}p_{k+1}+\varepsilon_{11}q_{k+1}=\varepsilon_{22}p_{k}+\varepsilon_{11}q_{k},\mbox{ for all }k\in\mathbb{N},\]
hence, if a limit $(p^*,q^*)$ exists, it satisfies
\[
r^*:=\varepsilon_{22}p^*+\varepsilon_{11}q^*=\varepsilon_{22}p_{0}+\varepsilon_{11}q_{0}=:r_0.\]
Furthermore, directly from the relation $f_1(p^*,q^*)=p^*$ we get the equality
\[
q^*=\frac{\varepsilon_{12}p^*}{\varepsilon_{11}+(\varepsilon_{12}-\varepsilon_{11})p^*}.
\]
Hence, the value of $(p^*,q^*)$ is given by the unique solution of the system
\begin{equation}
    \left\{\begin{matrix}
        r_0&=&\varepsilon_{22}p^*+\varepsilon_{11}q^*,\\
        &&\\
    q^*&=&\displaystyle{\frac{\varepsilon_{12}p^*}{\varepsilon_{11}+(\varepsilon_{12}-\varepsilon_{11})p^*}}.
    \end{matrix}\right.\label{pstar}
\end{equation}
Computing the Jacobian matrix at $(p^*,q^*)$ gives
\[
J_*:=J(p^*,q^*)=\begin{pmatrix}
    1-\varepsilon_{11}\frac{q^*}{p^*}&\varepsilon_{12}\frac{p^*}{q^*}\\
    \varepsilon_{22}\frac{q^*}{p^*}&1-\varepsilon_{21}\frac{p^*}{q^*}
\end{pmatrix}.
\]
The eigenvalues of $J_*$ are 
\[
\lambda^*_1=1-(\varepsilon_{11}\frac{q^*}{p^*}+\varepsilon_{21}\frac{p^*}{q^*})\mbox{ and }\lambda^*_2=1.
\]
Given that the second eigenvalue is equal to $1$, we cannot immediately give a conclusion to the stability of $(p^*,q^*)$. However, we can proceed as follows: for a fixed $(p_0,q_0)$, $p^*$ is solution of the equation
\begin{align*}
    p^*=&f_1(p^*,\frac{r^*-\varepsilon_{22}p^*}{\varepsilon_{11}})\\
    =&(1-\varepsilon_{12})p^*+(r^*-\varepsilon_{22}p^*)+(\varepsilon_{12}-\varepsilon_{11})p^*\frac{r^*-\varepsilon_{22}p^*}{\varepsilon_{11}}\\
    =&r^*+(1-(\varepsilon_{12}+\varepsilon_{22})+(\varepsilon_{12}-\varepsilon_{11})\frac{r^0}{\varepsilon_{11}})p^*+(\varepsilon_{22}-\varepsilon_{21})(p^*)^2\\
    =:&f(p^*).
\end{align*}
This is, $p^*$ is a fixed point of $f(p)$. Therefore, in order to determine the stability of $(p^*,q^*)$, it suffices to study the value of
\begin{align*}
    f'(p^*)&=(1-(\varepsilon_{12}+\varepsilon_{22})+(\varepsilon_{12}-\varepsilon_{11})\frac{r^0}{\varepsilon_{11}})+2(\varepsilon_{22}-\varepsilon_{21})p^*\\
    &=1-(\varepsilon_{11}\frac{q^*}{p^*}+\varepsilon_{21}\frac{p^*}{q^*}),
\end{align*}
which is precisely the first eigenvalue of $J^*$. Since $\lambda^*_1<1$, $(p^*,q^*)$ will be a stable fixed point if and only if $\lambda^*_1>-1$, or equivalently, if and only if
\[
g(p^*):=\frac{\varepsilon_{11}\varepsilon_{12}}{\varepsilon_{11}+(\varepsilon_{12}-\varepsilon_{11})p^*}+\frac{\varepsilon_{21}}{\varepsilon_{12}}(\varepsilon_{11}+(\varepsilon_{12}-\varepsilon_{11})p^*)=\varepsilon_{11}\frac{q^*}{p^*}+\varepsilon_{21}\frac{p^*}{q^*}<2.
\]
Since $g(p)$ is a convex function over $[0,1]$, which satisfies $g(0)=\varepsilon_{12}+\varepsilon_{22}<2$ and $g(1)=\varepsilon_{11}+\varepsilon_{21}<2$, we conclude $g(p^*)<2$ for all possible values of $p^*$. Therefore $(p^*,q^*)$ is a stable fixed point.
\end{proof}
Let us discuss the results from Proposition \ref{Stability}. There are two scenarios for the unbalanced case. If the players reaction to the lack of cooperation is stronger than the reaction to the presence of it ($e<0$), then both players will eventually adopt the no cooperation strategy, making the average expected gain equal to $0$. On the other hand, two players that are highly responsive to cooperation, and not to the lack of it ($e>0$), will eventually always cooperate, maximising this way the average expected gain.\\
We observe a far more complicated outcome when the responses of both players are balanced ($e=0$). Given that $(p^*,q^*)$ satisfies system \eqref{pstar}, then the average expected gain will increase if $\varepsilon_{11}<\varepsilon_{22}$ and the initial values $p_0$ and $q_0$ satisfy
\[q_0>\frac{\varepsilon_{12}p_0}{\varepsilon_{11}+(\varepsilon_{12}-\varepsilon_{11})p_0},\]
 or if $\varepsilon_{11}>\varepsilon_{22}$ and 
 \[q_0<\frac{\varepsilon_{12}p_0}{\varepsilon_{11}+(\varepsilon_{12}-\varepsilon_{11})p_0}.\]
 Thanks to the balance condition, $\varepsilon_{11}<\varepsilon_{22}$ implies that $\varepsilon_{12}<\varepsilon_{21}$. This is, in a way, player $A$ has more shy responses than player $B$. According to the previously established conditions, interactions between these two players will lead to an increase in the average expected gain only if the initial probability of cooperation for player $B$ is sufficiently big. An analogous interpretation can be given when $\varepsilon_{11}>\varepsilon_{22}$. Figure \ref{fig2} shows several initial configurations for $(p_0,q_0)$ and their respective limiting values satisfying the relation.
 \[
 q^*=\displaystyle{\frac{\varepsilon_{12}p^*}{\varepsilon_{11}+(\varepsilon_{12}-\varepsilon_{11})p^*}}.
 \]
\begin{figure}[H]
    \centering
    \hspace{-2cm}\includegraphics[width=18cm]{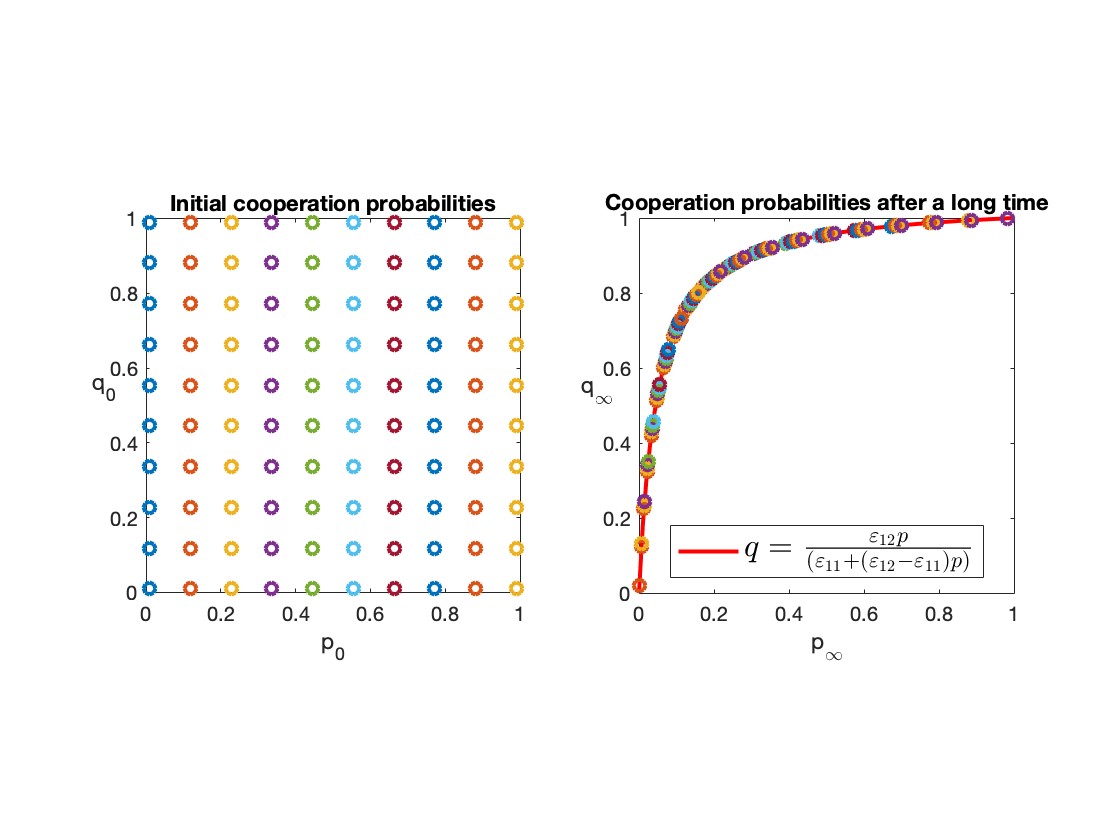}\vspace{-3cm}
    \caption{Left panel: Several initial configurations of cooperation probabilities. Right panel: Limiting values of the sequences $(p_k,q_k)$ associated to initial values showcased on the previous figure.}
    \label{fig2}
\end{figure}
Assume now the presence of $n$ players, each one with an initial probability of cooperation $p^i_0$ and reciprocity constants $(\varepsilon_{i1},\varepsilon_{i2})$, for $i\in\{1,\ldots,n\}$. The previous model can be adapted in such a way that each player modifies its strategy by taking into account the global cooperation level. This is, $p^i_k$ satisfies the relation
\begin{align*}
    p^i_{k+1}&=q^i_k(p_{k}+\varepsilon_{i1}(1-p^i_k))+(1-q^i_k)p^i_{k}(1-\varepsilon_{i2})\\
    &=(1-\varepsilon_{i2})p^i_k+\varepsilon_{i1}q^i_k+(\varepsilon_{i2}-\varepsilon_{i1})p^i_kq^i_k,
\end{align*}
with 
\[
q^i_k:=\frac{1}{n-1}\sum_{j\neq i}p^j_k,
\]
being the average probability of cooperation from the co-players of player $i$. As in the previous case, it can be expected that the amount of fixed points for this recurrence, and its stability will depend on a family of conditions over the values of $(\varepsilon_{i1},\varepsilon_{i2})$, however, for the moment being, we will not study this case any further.\\
An element that was not considered in these models was the effect of the average expected gain on the relation between $(p_k,q_k)$ and $(p_{k+1},q_{k+1})$. For example, considering variable reciprocity coefficients which directly depend on the average expected gain would create a mutual feedback between the cooperation probabilities and the gain, resulting this way in a far more complex, interesting and realistic model.
\subsection{A continuously structured population model for the evolution of cooperation}
Take $p\in[0,1]$ to be a continuous structure variable representing a probability of cooperation. Consider two populations $A$ and $B$, each one composed by individuals with different probabilities of cooperation with the elements on the other population. Let $n_A(t,p)$ and $n_B(t,p)$ be their respective population densities of individuals with probability of cooperation equal to $p$ at time $t$. The total populations at time $t$ are given by
\[
\rho_A(t):=\int_0^1n_A(t,p)dp\mbox{ and }\rho_B(t):=\int_0^1n_B(t,p)dp,
\]
and the mean cooperation probabilities by
\[
\Tilde{p}_A(t):=\frac{\int_0^1p n_A(t,p)dp}{\rho_A(t)}\mbox{ and }\Tilde{p}_B(t):=\frac{\int_0^1p n_B(t,p)dp}{\rho_B(t)}.
\]
These choices allow to define the global expected gain for each population. For the first population its global expected gain is defined then as
\[
E_A(t):=(b-c) \Tilde{p}_A(t) \Tilde{p}_B(t)+b(1-\Tilde{p}_A(t))\Tilde{p}_B(t)-c \Tilde{p}_A(t) (1-\Tilde{p}_B(t))=b\Tilde{p}_B(t)-c\Tilde{p}_A(t),
\]
where $b$ and $c$ are the benefit and cost, respectively, of cooperation in the prisoner's dilemma setting\footnote{For a more general model, the values of $b$ and $c$ could be dependent on $p$, this is, the cost and benefit of cooperation might depend on the probability of cooperation itself}. Similarly, the expected gain for population $B$ is given by
\[
E_B(t):=b\Tilde{p}_A(t)-c\Tilde{p}_B(t).
\]
This way, we may consider that the population densities evolve following the system of equations
\begin{equation}
    \left\{
    \begin{matrix}
        \partial_tn_A(t,p)+\varepsilon_A\partial_p\left((\Tilde{p}_B(t)-p)n_A(t,p)\right)=&g_A(p,E_A(t))n_A(t,p),\\
        &\\
        \partial_tn_B(t,p)+\varepsilon_B\partial_p\left((\Tilde{p}_A(t)-p)n_B(t,p)\right)=&g_B(p,E_B(t))n_B(t,p),\\
        &\\
        n_A(0,p)=n^0_A(p),\quad n_B(0,p)=n^0_B(p),
    \end{matrix}
    \right.\label{EvolCoop}
\end{equation}
where $\varepsilon_A$ and $\varepsilon_B$ are reciprocity coefficients and $g_A,g_B$ are continuous and increasing functions of $E_A$ and $E_B$ respectively, while the elements of both populations modify their probabilities of cooperation, depending on the global probability of cooperation of their counterpart.\\
There exists a formal link between the discrete model given in section \ref{Toy} and the partial differential equation system \eqref{EvolCoop} or some generalisation of it. Picture two populations, each one conformed by individuals that modify their probability of cooperation following the dynamics described by system \eqref{EvolCoop} by reciprocating the mean probability of cooperation of the opposing population. If we endow the system with appropriately chosen growth and death rates, that may be intrinsic or may depend on the global gain, it could be expected that, after taking the limit for a big number of individuals and a small time interval, the differential system of equations satisfied by the densities of population will resemble \eqref{EvolCoop}. We intend to study in depth this relation between the two models, discrete and continuous, in future work.\\
\paragraph{Cooperation or extinction, an easy choice} From model \eqref{EvolCoop}, we will illustrate, for an specific choice of $g_A$ and $g_B$, how cooperation may make a difference between extinction or persistence. Consider 
\begin{equation}
    \begin{matrix}
        g_A(p,E_A(t))&:=r_A(p)+\gamma_A(p)E_A(t)=r_A(p)+\gamma_A(p)(b\Tilde{p}_B(t)-c\Tilde{p}_A(t)),\\
    g_B(p,E_B(t))&:=r_B(p)+\gamma_B(p)E_B(t)=r_B(p)+\gamma_B(p)(b\Tilde{p}_A(t)-c\Tilde{p}_B(t)),
    \end{matrix}
\end{equation}
where $r_A(p),r_B(p)$ are the respective intrinsic growth rates of populations $A$ and $B$ and the non-negative functions $\gamma_A(p),\gamma_B(p)$ represent the effect of the expected gain on the growth rate of each population. This choice of $g_A$ and $g_B$ makes system \eqref{EvolCoop} bear a striking resemblance to the model studied in \cite{PoucholTrelat}, where conditions under which there is persistence of all species are given. Nevertheless, there are several differences: In our case the non local terms are given by the mean cooperation probabilities, the functions $\gamma_A$ and $\gamma_B$ are non-negative and we consider no restrictions over the signs of $r_A(p)$ and $r_B(p)$. Despite these differences, we do not rule out the fact that the tools and techniques used within the cited reference may be useful for the study of problem \eqref{EvolCoop} as well. For specific choices of $\varepsilon_A$, $\varepsilon_B$, $\gamma_A$ and $\gamma_B$ it is possible to identify the conditions over $r_A$, $r_B$, $b$ and $c$ which guarantee that one or both populations will either go extinct or proliferate. Such conditions are stated on the following proposition:
\begin{prop}\label{EorP}
    Consider $\varepsilon_A=\varepsilon_B=0$, $\gamma_A(p)\equiv \gamma_A$ and $\gamma_B(p)\equiv \gamma_B$, with $\gamma_A,\gamma_B$ non negative  constants. Suppose $r_A(p)$, $r_B(p)$, $n^0_A(p)$ and $n^0_B(p)$ to be continuous functions such that the maximum value of $r_A(p)$ over the support of $n^0_A(p)$ is attained at a single point $p^*_A$, and the maximum value of $r_B(p)$ over the support of $n^0_B(p)$ is attained at a single point $p^*_A$. Then
    \begin{itemize}
        \item[i)] If $r_A(p^*_A)+\gamma_A(bp^*_B-cp^*_A)<0$, population $A$ will go extinct.
        \item [ii)] If $r_A(p^*_A)+\gamma_A(bp^*_B-cp^*_A)>0$, there exists and interval $I$ satisfying $p^*_A\in I\subset [0,1]$ such that population $A$ will blow up for all $p\in I$.
        \item[iii)] The same is true for population $B$, depending on the sign of $r_B(p^*_B)+\gamma_B(bp^*_A-cp^*_B)$.
    \end{itemize}
\end{prop}
\begin{proof}
    Under these hypotheses, the expression for $n_A(p)$ and $n_B(p)$ are implicitly given by the expressions
    \[n_A(t,p)=n^A_0(p)e^{r_A(p)t+\gamma_A\int_0^tE_A(s)ds}\mbox{ and }n_B(t,p)=n^B_0(p)e^{r_B(p)t+\gamma_B\int_0^tE_B(s)ds},\]
    respectively. This allows to explicitly compute the values of $\Tilde{p}_A(t)$ and $\Tilde{p}_B(t)$:
    \[
    \Tilde{p}_A(t)=\frac{\int_0^1pn^0_A(p)e^{r_A(p)t}dp}{\int_0^1n^0_A(p)e^{r_A(p)t}dp}\mbox{ and }\Tilde{p}_B(t)=\frac{\int_0^1pn^0_A(p)e^{r_B(p)t}dp}{\int_0^1n^0_A(p)e^{r_B(p)t}dp}.
    \]
    From here, it is not hard to prove that, under the hypotheses of Proposition \ref{EorP}, $\Tilde{p}_A(t)$ and $\Tilde{p}_B(t)$ converge towards $p^*_A$ and $p^*_B$ respectively. This implies that, for all positive $\varepsilon$ there exists $T>0$ such that
    $r(p)+\gamma_A(b\Tilde{p}_B(t)-c\Tilde{p}_A(t))\leqslant r_A(p^*_A)+\gamma_A(bp^*_B-cp^*_A)+\varepsilon$ for all $t>T$. If $\varepsilon$ is chosen small enough, then $r(p)+\gamma_A(b\Tilde{p}^*_B(t)-c\Tilde{p}^*_A(t))<0$ for all $t>T$ which gives the convergence to $0$ of the population. Conversely, if $r_A(p^*_A)+\gamma_A(bp^*_B-cp^*_A)>0$, we set $\delta:=\frac{r_A(p^*_A)+\gamma_A(bp^*_B-cp^*_A)}{2}$, and define \[I=\{p\in[0,1]:r_A(p)>r_A(p^*_A)-\delta\}.\]
    Hence, for all $p\in I$ there exists $T>0$ such that
    \[r(p)+\gamma_A(b\Tilde{p}_B(t)-c\Tilde{p}_A(t))\geqslant r_A(p^*_A)-\delta +\gamma_A(b\Tilde{p}_B(t)-c\Tilde{p}_A(t))-\varepsilon=\delta-\varepsilon, \]
    for all $t>T$. Once again, by choosing $\varepsilon$ small enough we obtain the strictly positive growth rate for all values of $p\in I$, which implies the blow up of the population for all such values of $p$.\\
    The proof for population $B$ is analogous.
\end{proof}

Let us illustrate the result of Proposition \ref{EorP} with an example. Consider
\[
r_A(p)=r_B(p)=p(1-p)-\frac{1}{2}<0.
\]
It is straightforward to conclude that, if there is no cooperation ($\gamma_A(p)=\gamma_B(p)=0$ or $n^A_0(p)=n^B_0(p)=\rho_0\delta_0(p)$) then both populations will go extinct, at an exponential rate. On the other hand, consider $\gamma_A(p)=\gamma_B(p)=1$, $n^A_0(p)\equiv n^A_0$ and $n^B_0(p)\equiv n^B_0$. Under these assumptions, we have
\[
n_A(t,p)=n^A_0e^{r_A(p)t+\int_0^tE_A(s)ds}\mbox{ and }n_B(t,p)=n^B_0e^{r_B(p)t+\int_0^tE_B(s)ds},
\]
and consequently we get
\[
\Tilde{p}_A(t)=\frac{\int_0^1pe^{r_A(p)t}dp}{\int_0^1e^{r_A(p)t}dp}=\frac{1}{2}\mbox{ and }\Tilde{p}_B(t)=\frac{\int_0^1pe^{r_B(p)t}dp}{\int_0^1e^{r_B(p)t}dp}=\frac{1}{2},
\]
after integrating by means of a substitution. This way, the equations for $n_A(t,p)$ and $n_B(t,p)$ are reduced to
\begin{equation*}
    \left\{
    \begin{matrix}
        &\partial_tn_A(t,p)=(r_A(p)+\frac{(b-c)}{2})n_A(t,p),\\
        &\\
        &\partial_tn_B(t,p)=(r_B(p)+\frac{(b-c)}{2})n_B(t,p),\\
        &\\
       & n_A(0,p)=n^0_A,\quad n_B(0,p)=n^0_B.
    \end{matrix}
    \right.
\end{equation*}
It is then evident that, as long as $(b-c)>1$ there will be values of $p$ for which $r_A(p)+\frac{(b-c)}{2}>0$ and $r_B(p)+\frac{(b-c)}{2}>0$, hence, the population densities $n_A(t,p)$ and $n_B(t,p)$ will be proliferating exponentially.\\
An interesting question left unanswered is the case $r_A(p^*_A)+\gamma_A(bp^*_B-cp^*_A)=0$. In this scenario, additional conditions over the parameters of the problem might be needed in the general case in order to determine the behaviour of the solution. For the previous illustrative example, this condition is equivalent to choosing $b-c=\frac{1}{2}$, which leads to a solution which decreases  for all $p\neq 1/2$ and that remains constant for $p=1/2$.\\
When $\gamma_A(p)$ and $\gamma_B(p)$ are constant, the conditions determining the fate, blow-up or extinction, of the population depends on the values of $p$ maximising the intrinsic growth rates $r_A(p)$ and $r_B(p)$. Furthermore, we showed that the mean cooperation probabilities of both populations converge to said values of $p$, respectively. This last property reminisces of the concentration result shown in Section 2.1 of \cite{PerthBook} where, for a constant death rate, the density of the structured population converges to a Dirac delta centred around the value that maximises the growth rate. When said death rate is not constant, it has been already proved in \cite{PoucholTrelat} that the concentration phenomena still occurs, but around the value maximising the fitness function defined as the quotient between the growth and death rates. One could conjecture that, in our case, for non constant choices of $\gamma_A(p)$ and $\gamma_B(p)$, the mean probabilities of cooperation will still converge towards the values maximising a conveniently defined fitness function. Additionally, if $\varepsilon_A$ and $\varepsilon_B$ are non-zero, the long time behaviour is not that clear, but we must remark that, for similar models containing advection terms, the theoretical results from \cite{Guilberteau} and numerical evidence in \cite{ACC2022} suggest that concentration phenomena around values of $p$ depending on the drift could also be expected.\\
A diffusion term can be considered as well in both equations of system \eqref{EvolCoop} in order to model random instabilities of the probabilities of cooperation. This term may be a second order differential operator and suitable boundary conditions, or an integral term with a mutation kernel. A similar model, for only one population, excluding the advection term and depending on the population size as the non-local term was already studied in \cite{LorenziPouchol}. Finally, if all three terms are considered, the resulting model will follow the same principles as in model \eqref{PheD}, where the diffusion term represents the non-genetic instability of trait $p$, the advection term represents the external stress exerted over each population as in \cite{ACC2022}  or the existence of a bias in the direction of epimutations, as in \cite{Chisholm2016} (in our case such stress or bias is prompted by the global cooperation probability of the other population)  and the reaction term accounts for selection mechanisms.It is worth mentioning that a particle method allowing for the numerical approximation of solutions for models accounting for the three previously mentioned mechanisms was recently developed in \cite{AlvarezGuilberteau}.

\section{Conclusion}

We have sketched in this short essay, relying on concepts of philosophy of science and on  mathematical models under development, the two settings of evolution in which phenotype divergence and cooperation between phenotypes in the constitution of animal multicellularity should be considered from our point of view. They are the billion-year Darwinian evolution of species - which we assimilate with the evolution of body plans - and the short-term construction, in embryogenesis and development, of an isogenic animal from the zygote to the constituted, terminally differentiated multicellular organism.

In the first case, phenotype divergence is considered to be determined by changes in the environment, and it is represented by an advection term in a PDE, yielding different optimal adaptive strategies that are chosen randomly in the initial body plan and resulting in (at least) two different body plans, that in the first place should be reversible, before being fixed by stabilising mutations.

In the second case, the body plan of a given coherent multicellular animal, that has been established in Darwinian evolution in a deterministic machinery of embryogenesis and organism maintenance, governs the process of development from the zygote of the animal individual on  principles of compatibility and cooperativity between physiological functions, organs and tissues, that relies on cell differentiations. Of note, cellular stress-induced genes might evolve into developmental organisers, according to a mechanism proposed in the Chlamydomonas/Volvox lineage~\cite{Konig2020}. Such differentiations are by nature theoretically reversible, relying on epigenetic enzyme activities which graft methyl or acetyl radicals on the DNA or on the histones that constitute the genome on animal, and dedifferentiations indeed have been shown to be experimentally possible in 2006 by Takahashi and Yamanaka~\cite{Yamanaka2006}. However, they are physiologically excluded, except in particular situations such as wound healing, by a strict 
control of the expression of these epigenetic enzymes. Plasticity in cancer cells alters such normal organismic control.

In cancer, which is a disease characteristic of multicellular animals, differentiations are (locally, in the tissue from which it originates) out of organismic control, so that tumours, as poorly organised cell colonies that nevertheless are made of cells bearing in each one of them the body plan of a multicellular organism, can reactivate a process of phenotype divergence in response to a deadly insult (such as a chemotherapy at high doses), resulting in cancer bet hedging, i.e., developing diverse transient (reversible) phenotypes without organised control, with the goal to preserve the proliferation potential of their cells.

We are aware that the mathematical models presented here are sketches that need refinement, and that in particular the cooperativity part should be oriented towards defining a compulsory common gain (likely represented by, again, an advection term in a  PDE) that determines the precise construction of an individual animal organism designed by its body plan. Much still remains to be done towards this goal, and in particular the body plan - whose effects are patent in embryogenesis and development, but is still not properly defined as a programme - needs to be better defined in a mathematical representation. It is likely made of an organised ensemble of gene regulatory networks, as evidenced in the works of Eric Davidson~\cite{Davidson1995} and his colleagues, and systematically described in the diversity of its functions in hypothetical {\it Urmetazoa} by W.E.G. M\"uller and his colleagues~\cite{MulleretalIntRevCytol2004}. A mathematical representation of the body plan, as a programme of construction of the individual and as the evolutionary unit on which relies Darwinian evolution and the design of animal anatomy and physiology, is a challenge that awaits philosophers, evolutionary biologists, and mathematical modellers and analysts, a challenge we have merely sketched in this short essay.

\end{document}